%% file: Meromorphic.tex
\documentclass[a4paper]{amsart}
\usepackage[leqno]{amsmath}
\usepackage{amssymb}
\usepackage{amscd}
\usepackage{amsthm}
\usepackage{mathrsfs}
\usepackage{mathtools}
\usepackage{bbm}
\usepackage{color}
\usepackage{enumerate}
\usepackage{cite}
\usepackage[utf8]{inputenc}
\usepackage[all,cmtip]{xy}
\usepackage{etoolbox}
\usepackage{tikz}
\usepackage{extarrows}

\numberwithin{equation}{section}

\newcommand{\Z}{\ensuremath{\mathbb{Z}}}
\newcommand{\Q}{\ensuremath{\mathbb{Q}}}
\newcommand{\R}{\ensuremath{\mathbb{R}}}
\newcommand{\C}{\ensuremath{\mathbb{C}}}

\DeclareMathOperator{\Hom}{Hom}

\DeclareMathOperator{\SL}{SL}
\DeclareMathOperator{\ord}{ord}

\DeclareMathOperator{\Coind}{Coind}

\DeclareMathOperator{\res}{res}

\DeclareMathOperator{\qd}{qd}

\DeclareMathOperator{\HH}{H}

\DeclareMathOperator{\Dr}{\mathcal{H}}

\newcommand{\BS}{\ensuremath{\mathbb{P}_B}}
\newcommand{\PP}{\ensuremath{\mathbb{P}^1}}

\DeclareMathOperator{\red}{red}

\DeclareMathOperator{\Div}{Div}
\DeclareMathOperator{\supp}{supp}

\newcommand{\p}{\ensuremath{\mathfrak{p}}}
\newcommand{\q}{\ensuremath{\mathfrak{q}}}

\newcommand{\into}{\hookrightarrow}

\newcommand{\too}{\longrightarrow}								
\newcommand{\mapstoo}{\longmapsto}

\newtheorem{Lem}{Lemma}
\makeatletter
\newlength{\@thlabel@width}%
\newcommand{\thmenumhspace}{\settowidth{\@thlabel@width}{\itshape1.}\sbox{\@labels}{\unhbox\@labels\hspace{\dimexpr-\leftmargin+\labelsep+\@thlabel@width-\itemindent}}}
\makeatother
\newtheorem{Pro}[Lem]{Proposition}

\newtheorem{Thm}[Lem]{Theorem}

\newtheorem{Cor}[Lem]{Corollary}

\theoremstyle{definition}
\newtheorem{Def}[Lem]{Definition}
\newtheorem{Rem}[Lem]{Remark}

\author[L. Gehrmann]{Lennart Gehrmann}
\address{L. Gehrmann \\ Fakult\"at f\"ur Mathematik \\ Universit\"at Duisburg-Essen \\ Thea-Leymann-Stra\ss e 9 \\ 45127 Essen \\ Germany}
\email{lennart.gehrmann@uni-due.de}

\title[On quaternionic rigid meromorphic cocyles]{On quaternionic rigid meromorphic cocycles}
\subjclass[2010]{}

\begin{document}

\begin{abstract}
Recently, Darmon and Vonk initiated the theory of rigid meromorphic cocycles for the group $\SL_2(\Z[1/p])$.
One of their major results is the algebraicity of the divisor associated to such a cocycle.
We generalize the result to the setting of $\p$-arithmetic subgroups of inner forms of $\SL_2$ over arbitrary number fields.
The method of proof differs from the one of Darmon and Vonk.
Their proof relies on an explicit description of the cohomology via modular symbols and continued fractions, whereas our main tool is Bieri-Eckmann duality for arithmetic groups.
\end{abstract}

\maketitle

\tableofcontents

\subsection*{Introduction}
In \cite{DV} Darmon and Vonk initiate the theory of $p$-adic singular moduli for real quadratic fields.
In this theory classical modular functions such as the $j$-invariant are replaced by so-called rigid meromorphic cocycles.
These are $\SL_2(\Z[1/p])$-invariant modular symbols with values in rigid meromorphic functions on Drinfeld's $p$-adic upper half plane. 
One of their first results (see Theorem 1 of \textit{loc.~cit.}) states that the divisor of a rigid meromorphic cocycle is supported on finitely many $\SL_2(\Z[1/p])$-orbits of real quadratic points, i.e.~ points which generate real quadratic extensions of $\Q$.
This highly suggests that rigid meromorphic cocyles are a real quadratic analogue of Borcherds' singular theta lifts of modular forms of weight $1/2$. 

It is a natural question to ask whether this result is true in more general situations. For instance, one may replace the group $\SL_2(\Z[1/p])$ by a $p$-arithmetic congruence subgroup, consider arbitrary ground fields and inner forms of $\SL_2$.
The proof of Darmon and Vonk relies on rather explicit methods, i.e.~it uses explicit generators of $\SL_2(\Z)$ to reduce the question to one about continued fractions following an approach of Choie--Zagier (cf.~\cite{ChZa}). 
This approach does not generalize easily to a more general setup.
The aim of this note is to prove the algebraicity of divisors in the before mentioned general situation by purely cohomological methods.\footnote[1]{The approach taken in this article is closely related to the classification of rational period functions by Ash (see \cite{Ash}).}

The statement becomes a bit more involved:
the divisor of a rigid meromorphic cocycle is still supported on finitely many orbits of quadratic points, i.e.~points that generate quadratic extension of the ground field.
But their splitting behaviour at infinity depends on the inner form of $\SL_2$ and the degree of cohomology we are working with.
For example, let us suppose that the ground $F$ is totally real and let $\left\{\infty_1,\ldots,\infty_d\right\}$ be the set of Archimedean places of $F$.
Assume that the inner form of $\SL_2$ is split at $\infty_1,\ldots,\infty_r$ and non-split at $\infty_{r+1},\ldots,\infty_d$.
Then, the divisor of a meromorphic cocycle in middle degree cohomology is supported on orbits of points, which generate quadratic extensions in which $\infty_1,\ldots,\infty_r$ split and $\infty_{r+1},\ldots,\infty_d$ do not.
If we move beyond middle degree, quadratic extensions in which $\infty_1,\ldots,\infty_r$ are non-split appear as well. 

As a by-product of our results we notice the following surprising fact:
if one wants to generalize Darmon and Vonk's theory of singular moduli one is forced to work with totally real ground fields. 

We hope that our more conceptual approach will pave the way to a generalization of the theory to groups beyond $\SL_2$.

\subsection*{Acknowledgements}
While working on this manuscript I was visiting McGill University, supported by Deutsche Forschungsgemeinschaft, and I would like to thank these institutions.
I am grateful to Xavier Guitart, Marc Masdeu and Xavier Xarles for pointing out a mistake in an earlier draft of the article.
I thank Henri Darmon, Mathilde Gerbelli-Gauthier and the anonymous referee for their comments, which helped to improve the exposition.
It is my pleasure to thank Markus Severitt, who explained to me what a Brauer-Severi variety is a long time ago.

\subsection*{Setup}
We fix an algebraic number field $F$ and a finite place $\p$ of $F$ lying above the rational prime $p$.
Let $F_\p$ be the completion of $F$ at $\p$ with ring of integers $\mathcal{O}_\p.$

In addition, we fix a quaternion algebra $B$ over $F$ that is split at $\p$.
Let $G$ be the the group of elements of reduced norm one of $B$ viewed as an algebraic group over $F$ and put $G_\p=G(F_\p)$.
We denote by $\BS$ the Brauer-Severi variety associated to $B$, i.e.~for every field extension $E/F$ we have
$$\BS(E)=\left\{I\unlhd B\otimes_F E \mbox{ left ideal}\mid \dim_E I= 2\right\}.$$
It is a smooth curve over $F$ such that $\BS(E)\neq\emptyset$ if and only if $E$ is a splitting field of $B$.
In that case we can choose compatible non-canonical isomorphisms $\mathbb{P}_{B,E}\cong \mathbb{P}^1_E$ and $G_E\cong \SL_{2,E}$. 
If $B$ is already split, we will identify $\BS$ with $\mathbb{P}^{1}_F$ and $G$ with $\SL_{2,F}.$

Let $r_1(B)$ be the number of real places of $F$ at which $B$ is split and $r_2$ the number of complex places of $F$ and put $r=r_1(B)+r_2.$
Further, we put $d=2\cdot r_1(B)+3\cdot r_2$.
This is the real dimension of the symmetric space associated to $G(F\otimes_\Q \R)$.

\section{The $p$-adic upper half plane}
\input{Mer-Divisors.tex}
\input{Mer-Functions.tex}

\section{Rigid meromorphic cocyles}
\input{Mer-Cohomology.tex}

\input{Mer-Cocycles.tex}

\input{Mer-Moduli.tex}

\bibliographystyle{abbrv}
\bibliography{bibfile}

\end{document}

%% file: Mer-Divisors.tex
\subsection{Divisors on the $p$-adic upper half plane}\label{Divisors}
Let $\C_p$ be the completion of a fixed algebraic closure of $F_\p$.
Let us remind ourselves that the $p$-adic upper half plane $\Dr_p$ associated to $F_\p$ is the rigid analytic space over $\C_p$ whose $\C_p$-valued points are given by
$$\Dr_{p}(\C_p)=\PP(\C_p)\smallsetminus \PP(F_{\p}).$$

We will work with the following coordinate-free version:
the $p$-adic upper half plane $\Dr_{B_\p}$ associated to $G_{F_\p}$ is the rigid analytic space over $\C_p$ such that
$$\Dr_{B_\p}(\C_p)=\BS(\C_p)\smallsetminus \BS(F_{\p}).$$
It is non-canonically isomorphic to $\Dr_{p}$ but carries a canonical action of the group $G_\p$.

Given a rigid analytic subvariety $X\subseteq \Dr_{B_\p}$ we often also write $X$ for the set of $\C_p$-valued points of $X$.

Given a subset $X\subseteq \Dr_{B_\p}$ we denote by $\Div(X)$ the space of all maps from $X$ to $\Z$ with finite support.
Further, we define $\Div^{\dagger}(X)$ to be the space of locally finite divisors on $X$, that is, the space of all maps $D\colon X\to \Z$ such that for every affinoid subvariety $Y\subseteq \Dr_{B_\p}$ the restriction $\left.D\right|_{(X\cap Y)}$ is an element of $\Div (X\cap Y)$.
 
Let $\mathcal{T}=(\mathcal{V},\mathcal{E})$ be the unoriented Bruhat-Tits tree of $G_\p$, i.e.
\begin{itemize}
\item $\mathcal{V}_{\mathcal{T}}$ is the set of maximal orders of $B_{F_\p}$ and
\item there is an edge in $\mathcal{E}_{\mathcal{T}}$ connecting $v,v^{\prime}\in\mathcal{V}$ if and only if the intersection of the corresponding orders is an Eichler order of level $\p$.
\end{itemize}
We denote by $\red\colon\Dr_{B_\p}(\C_p)\to \mathcal{T}$ the $G_\p$-equivariant reduction map from the $p$-adic upper half plane to the tree.
Suppose $D\in\Div^{\dagger}(\Dr_{B_\p})$ is a locally finite divisor on $\Dr_{B_\p}$.
Then $\left.D\right|_{\red^{-1}(v)}$ (respectively $\left.D\right|_{\red^{-1}(e)}$) is an element of $\Div (\red^{-1}(v))$ (respectively $\Div (\red^{-1}(e))$) for any vertex $v$ (respectively edge $e$) of $\mathcal{T}.$ 
Given a vertex $v$ (respectively edge $e$) of $\mathcal{T}$ we denote its stabilizer by $K_{\p,v}$ (respectively $K_{\p,e}$).

\begin{Lem}\label{coind}
Let $v$ and $\bar{v}$ be two adjacent vertices of $\mathcal{T}$ with connecting edge $e$.
There exists a canonical isomorphism of $G_\p$-modules:
\begin{align*}\Div^{\dagger}(\Dr_{B_\p}) \cong &\Coind_{K_{\p,v}}^{G_\p}\Div(\red^{-1}(v))\
 \oplus\ \Coind_{K_{\p, \bar{v}}}^{G_\p}\Div(\red^{-1}(\bar{v}))\\
 \oplus &\Coind_{K_{\p,e}}^{G_\p}\Div(\red^{-1}(e)).
\end{align*}
\end{Lem}
\begin{proof}
Let us remind ourselves that the coinduction of $\Div(\red^{-1}(v))$ from $K_{\p,v}$ to $G_\p$ is the space of all functions $f\colon G_\p \to \Div(\red^{-1}(v))$ such that $f(kg)=k.(f(g))$ holds for all $k\in K_{\p,v}$ and $g\in G_\p.$
It is easy to see that the map
$$
\Coind_{K_{\p,v}}^{G_\p}\Div(\red^{-1}(v)) \too \Div^{\dagger}(\Dr_{B_\p}),\quad f\mapstoo \sum_{g\in K_{\p,v}\backslash G_\p} g^{-1}.(f(g))
$$
is well-defined, injective and $G_\p$-equivariant.
Similarly we can define injective $G_\p$-equivariant maps
\begin{align*}
\Coind_{K_{\p,\bar{v}}}^{G_\p}\Div(\red^{-1}(\bar{v})) &\too \Div^{\dagger}(\Dr_{B_\p})\\
\intertext{and}
\Coind_{K_{\p,e}}^{G_\p}\Div(\red^{-1}(e)) &\too \Div^{\dagger}(\Dr_{B_\p}).
\end{align*}
The claim follows since the set $\{v, \bar{v},e\}$ is a fundamental domain for the action of $G_\p\cong\SL_2(F_\p)$ on $\mathcal{T}$ (see for example \cite{Trees}, Section II.1.4, Theorem 2).
\end{proof}

%% file: Mer-Functions.tex
\subsection{Functions on the $p$-adic upper half plane}\label{Functions}
In this section we recall well-known facts about rigid analytic functions on $\Dr_{B_\p}.$
Let $\mathcal{M}^{\times}$ the multiplicative group of rigid meromorphic functions on $\Dr_{B_\p}$ and $\mathcal{A}^{\times}\subseteq\mathcal{M}^{\times}$ the subgroup of invertible rigid analytic functions.
For a meromorphic function $f\in\mathcal{M}^{\times}$ and a point $z\in\Dr_{B_\p}(\C_p)$ the order of vanishing of $f$ at $z$ is denoted by $\ord_z(f)\in\Z$.
The map
$$\Div\colon\mathcal{M}^{\times}\too \Div^{\dagger}(\Dr_{B_\p}),\quad f \mapstoo [z\mapsto \ord_z(f)]$$
is well-defined and $G_\p$-equivariant with respect to the natural action on both sides.
\begin{Pro}\label{divseq}
The sequence
$$0\too\mathcal{A}^{\times}\too\mathcal{M}^{\times}\too \Div^{\dagger}(\Dr_{B_\p}) \too 0$$
is exact.
\end{Pro}
\begin{proof}
This is Proposition 2.2 of \cite{vdP2}.
\end{proof}

Let $X\subseteq \Dr_{B_\p}$ be a subset.
We define $\mathcal{M}^{\times}\hspace{-0.2em}(X)\subseteq\mathcal{M}^{\times}$ as the subgroup of those functions whose divisor is supported on $X$.
The following claim follows immediately from the proposition above.
\begin{Cor}
The sequence
$$0\too\mathcal{A}^{\times}\too\mathcal{M}^{\times}\hspace{-0.2em}(X)\too \Div^{\dagger}(X) \too 0$$
is exact.
\end{Cor}

We fix an orientation on the edges of $\mathcal{T}$.
Note that the action of $G_\p$ preserves any such orientation.
Therefore, the surjective map
$$
\partial\colon C(\mathcal{E}_\mathcal{T},\Z)\too C(\mathcal{V}_\mathcal{T},\Z),\quad f \mapstoo [v\mapsto \sum_{s(e)=v} f(e) - \sum_{t(e)=v} f(e)]
$$
is $G_\p$-equivariant.
Its kernel consists of the so-called $\Z$-valued harmonic cochains on the tree.

In \cite{vdP}, Section 1, van der Put constructs a $G_\p$-equivariant map
$$\mathcal{A}^\times\hspace{-0.2em}/\C_p^{\times} \too C(\mathcal{E}_\mathcal{T},\Z)$$ and proves the following proposition (cf.~\cite{vdP}, Proposition 1.11. See also \cite{FvdP}, Theorem 2.7.11).
\begin{Pro}\label{first}
The sequence
$$0 \too \mathcal{A}^{\times}\hspace{-0.2em}/\hspace{0.1em}\C_p^{\times} \too C(\mathcal{E}_\mathcal{T},\Z)\xlongrightarrow{\partial} C(\mathcal{V}_\mathcal{T},\Z) \too 0$$
of $G_\p$-modules is exact.
\end{Pro}

\begin{Rem}
The isomorphism of the space of invertible functions modulo constants with a space of harmonic cochains on the Bruhat-Tits building has recently been generalized by Gekeler to the case of Drinfeld's upper half space in higher dimension (see \cite{Gekeler}, Theorem 3.11).
\end{Rem}

Let $A$ be an abelian group equipped with the trivial action by $G_\p$.
As before, we fix an orientation on the edges of $\mathcal{T}$.
We define the $G_\p$-equivariant homomorphism
$$d\colon C(\mathcal{V}_\mathcal{T},A)\too C(\mathcal{E}_\mathcal{T},A),\quad f \mapstoo [e\mapsto f(t(e))-f(s(e))].$$
Further, we consider the embedding
$$i\colon A\too C(\mathcal{V}_\mathcal{T},A),\quad a \mapstoo [v\mapsto a].$$
The following lemma can be deduced from the fact that the Bruhat-Tits tree is contractible. 
\begin{Lem}\label{second}
The sequence
$$0 \too A \xlongrightarrow{i} C(\mathcal{V}_\mathcal{T},A)\xlongrightarrow{d} C(\mathcal{E}_\mathcal{T},A) \too 0$$
of $G_\p$-modules is exact.
\end{Lem}

\begin{Rem}
Let $v$ and $\bar{v}$ be two adjacent vertices of $\mathcal{T}$ with connecting edge $e$.
For every abelian group there are $G_\p$-equivariant isomorphisms
\begin{align*}
C(\mathcal{E}_\mathcal{T},A)&\cong \Coind_{K_{\p,e}}^{G_\p} A\\
\intertext{and}
C(\mathcal{V}_\mathcal{T},A)&\cong \Coind_{K_{\p,v}}^{G_\p} A\ \oplus\ \Coind_{K_{\p,\bar{v}}}^{G_\p} A.
\end{align*}
\end{Rem}

%% file: Mer-Cohomology.tex
\subsection{Divisor valued cohomology classes}\label{Cohomology}
We define $\Delta_0$ as the kernel of the map $$\Z[\PP(F)] \to \Z,\quad \sum_P m_P P \mapsto \sum_P m_P.$$
The $\SL_2(F)$-action on $\PP(F)$ induces an action on $\Delta_0$.
Let $\Gamma\subseteq G(F)$ be an arbitrary subgroup and $A$ a $\Z[\Gamma]$-module.
If $G$ is split, we put
$$\HH^{i}_c(\Gamma,A)=\HH^{i-1}(\Gamma,\Hom_\Z(\Delta_0,A)).$$
If $G$ is non-split, we simply set $\HH^{i}_c(\Gamma,A)=\HH^{i}(\Gamma,A)$.

Let us fix a $\p$-arithmetic congruence subgroup $\Gamma^{\p}\subseteq G(F)$ and a $\Gamma^{\p}$-stable subset $X\subseteq \Dr_{B_\p}$.
For every $\Gamma^\p$-orbit $\Gamma^{\p}x\subseteq X$ the inclusion $\Div^{\dagger}(\Gamma^{\p}x)\into \Div^{\dagger}(X)$ induces a map
$$\HH^{i}_c(\Gamma^{\p},\Div^{\dagger}(\Gamma^{\p}x))\too \HH^{i}_c(\Gamma^{\p},\Div^{\dagger}(\Dr_{B_\p}))$$
in cohomology.
\begin{Pro}\label{new}
The canonical map
\begin{equation}\label{decom}
\bigoplus_{\Gamma^{\p}x \in \Gamma^{\p}\backslash X} \HH^{i}_c(\Gamma^{\p},\Div^{\dagger}(\Gamma^{\p}x)) \too \HH^{i}_c(\Gamma^{\p},\Div^{\dagger}(X)) \tag{+}.
\end{equation}
is an isomorphism.
\end{Pro}
\begin{proof}
The action of $\Gamma_\p$ on $\mathcal{V}\cup \mathcal{E}$ has only finitely many orbits.
We denote these by $o_1,\ldots,o_h$ and put $X_j=X\cap\red^{-1}(o_j)$ for $j\in\{1,\ldots,h\}$.
It is enough to prove that the canonical map
$$\bigoplus_{\Gamma^{\p}x \in \Gamma^{\p}\backslash X_j} \HH^{i}_c(\Gamma^{\p},\Div^{\dagger}(\Gamma^{\p}x))
\too \HH^{i}_c(\Gamma^{\p},\Div^{\dagger}(X_j))$$
is an isomorphism for all $j\in\{1,\ldots,h\}$.

Let $v_j$ be an element of $o_j$ (note that $v_j$ is either a vertex or an edge).
We write $\Gamma_{v_j}$ for the stabilizer of $v_j$ in $\Gamma^\p$ and put $X_{v_j}=X\cap \red^{-1}(v_j).$
As in the proof of Lemma \ref{coind} one can prove that there exists an isomorphism
$$\Coind_{\Gamma_{v_j}}^{\Gamma^\p}\Div(X_{v_j})\xlongrightarrow{\cong} \Div^{\dagger}(X_j).$$
of $\Gamma^\p$-modules.
Therefore, Shapiro's Lemma yields an isomorphism
$$\HH^{i}_c(\Gamma^{\p},\Div^{\dagger}(X_j))\xlongrightarrow{\cong} \HH^{i}_c(\Gamma_{v_j},\Div(X_{v_j})).$$
As a $\Gamma_{v_j}$-module the space of divisors $\Div(X_{v_j})$ decomposes as follows
$$
\Div(X_{v_j})=\bigoplus_{\Gamma_{v_j} x \in \Gamma_{v_j}\hspace{-0.2em}\backslash X_{v_j}} \Div(\Gamma_{v_j} x).
$$
Since arithmetic groups are of type \emph{(VFL)} by a theorem of Borel and Serre (see \cite{BS}, Section 11.1), the functor $N \mapsto \HH^{i}_c(\Gamma_{v_j},N)$ commutes with direct limits (cf.~\cite{Se2}, p.~101).
Thus, the canonical map
$$
\bigoplus_{\Gamma_{v_j} x \in \Gamma_{v_j}\hspace{-0.2em}\backslash X_{v_j}} \HH^{i}_c(\Gamma_{v_j},\Div(\Gamma_{v_j} x))\xlongrightarrow{\cong}\HH^{i}_c(\Gamma_{v_j},\Div(X_{v_j}))$$
is an isomorphism.

For every $\Gamma^\p$-orbit in $X_j$ we may choose a representative $x\in X_{v_j}$.
As above Shapiro’s Lemma yields a canonical isomorphism
$$\bigoplus_{\Gamma^{\p}x \in \Gamma^{\p}\backslash X_j} \HH^{i}_c(\Gamma^{\p},\Div^{\dagger}(\Gamma^{\p}x))
\xlongrightarrow{\cong}
\bigoplus_{\Gamma_{v_j} x \in \Gamma_{v_j}\hspace{-0.2em}\backslash X_{v_j}} \HH^{i}_c(\Gamma_{v_j},\Div(\Gamma_{v_j} x)),$$
which proves the claim.
\end{proof}

\begin{Def}
Let $D\in \HH^{i}_c(\Gamma^{\p},\Div^{\dagger}(\Dr_{B_\p}))$ be a divisor-valued cohomology class.
Its \emph{support} $\supp(D)$ is the union of all $\Gamma^{\p}$-orbits $\Gamma^{\p}x\subseteq \Dr_{B_\p}$ such that the restriction of $D$ to $\HH^{i}_c(\Gamma^{\p},\Div^{\dagger}(\Gamma^{\p}x))$ with respect to the decomposition \eqref{decom} is non-zero.
\end{Def}
Note that from the decomposition \eqref{decom} we deduce that the support of a divisor-valued cohomology class is always a finite union of $\Gamma^{\p}$-orbits.

Given a subset $X\subseteq \Dr_{B_\p}$ we define 
$$X^{\qd}=\left\{x\in X\mid x\in \BS(F(x))\ \mbox{for some quadratic extension}\ F(x)/F \right\}.$$

The following lemma is probably a well-known statement.
But since we could not find it in the literature we included a proof.
\begin{Lem}\label{groups}
An element $x\in \Dr_{B_\p}$ belongs to $\Dr_{B_\p}^{\qd}$ if and only if the stabilizer $G(F)_x$ of $x$ in $G(F)$
has a non-torsion element.
In that case we have: 
\begin{enumerate}[(i)]
\item $F(x)$ is a splitting field of $B$.
\item There exists an embedding $F(x)\into B$, which induces an isomorphism between the group of norm one elements of $F(x)$ and $G(F)_x$.
\item The prime $\p$ is non-split in $F(x)$.
\end{enumerate}
\end{Lem}
\begin{proof}
If $x$ belongs to $\Dr_{B_\p}^{\qd}$, then by definition $F(x)$ is a splitting field of $B$ (and $\p$ is non-split in $F(x)$).
The Weil restriction of $\mathbb{P}_{B,F(x)}$ from $F(x)$ to $F$ is a two-dimensional $F$-variety, on which the three-dimensional $F$-algebraic group $G$ acts algebraically.
(Note that $\mathbb{P}_{B,F(x)}$ is non-canonically isomorphic to $\mathbb{P}^{1}_{F(x)}$.)
Therefore, the stabilizer $G_x$ of $x$ viewed as an $F$-algebraic group is at least one-dimensional.
Since by definition $x$ is not an element of $\mathbb{P}_{B}(F)$, one deduces that $G_x$ has to be a non-split torus, which splits after base change to $F(x)$.
Thus, it is isomorphic to the group of norm one elements in $F(x)$ and, hence, the group $G(F)_x$ has a non-torsion point.

Conversely, let us assume that $G(F)_x$ has a non-torsion element $\gamma$.
By construction $\gamma$ is neither unipotent nor an element of an $F$-split torus of $G$.
Thus, the quadratic extension $F(\gamma)\subseteq B$ generated by $\gamma$ is a splitting field of $B$.
Hence, $\gamma$ is a non-torsion element inside a split torus inside $G_{F(\gamma)}\cong \SL_{2,F(\gamma)}.$
It is easy to see (for example by conjugating to the diagonal torus) that for any field extension $E/F(x)$ the element $\gamma$ has exactly two fixed points on $\mathbb{P}_{B}(E)\cong\mathbb{P}^{1}_{F(\gamma)}(E)$ both of which are defined over $F(\gamma)$.
\end{proof}

Let $X$ be a subset of $\Dr_{B_\p}$ and $x$ an element of $X^{\qd}$.
We write $S_{\infty,sp}(x)$ for the set of Archimedean places of $F$ which split in $F(x)$.
Since $F(x)$ can be embedded into $B$, it follows that $S_{\infty,sp}(x)$ is a subset of those Archimedean places at which $B$ is split.
In particular, we have $|S_{\infty,sp} (x)| \leq r= r_1(B) + r_2.$
For integers $s\geq s^{\prime}\geq 1$ we define
\begin{align*}
X^{\qd,s}&=\left\{x\in X^{\qd}\mid |S_{\infty,sp} (x)|=s \right\}\\
\intertext{and}
X^{\qd,[s^{\prime},s]}&=\left\{x\in X^{\qd}\mid |S_{\infty,sp} (x)|\in [s^{\prime},s] \right\}.
\end{align*}

\begin{Pro}\label{Theorem}
Let $\Gamma^{\p}\subseteq G(F)$ be a $\p$-arithmetic congruence subgroup.
\begin{enumerate}[(i)]
\item\label{T1} Let $s\geq 1$ be an integer and $D\in \HH^{d-s}_c(\Gamma^{\p},\Div^{\dagger}(\Dr_{B_\p}))$ a divisor-valued cohomology class.
We have
$$\supp(D)\subseteq \Dr_{B_\p}^{\qd,[s,r]}.$$
In particular,
$\HH^{i}_c(\Gamma^{\p},\Div^{\dagger}(\Dr_{B_\p}))= 0$
for $i< d-r=r_1(B)+2\cdot r_2$.
\item\label{T2} We have
$$\HH^{d-r}_c(\Gamma^{\p},\Div^{\dagger}(\Dr_{B_\p}))\cong \bigoplus_{\Gamma^{\p}\backslash \Dr_{B_\p}^{\qd,r}} \Z.$$
\item\label{T3} In the range $d-r\leq i \leq d-1$ the torsion-free part of
$\HH^{i}_c(\Gamma^{\p},\Div^{\dagger}(\Dr_{B_\p}))$
is free of countably infinite rank.
\end{enumerate}
\end{Pro}
\begin{proof}
The group $\Gamma^{\p}$ has a torsion-free, normal subgroup of finite index.
By applying the Hochschild-Serre spectral sequence all claims can be reduced to the case of a torsion-free group, which is the content of the next lemma.
\end{proof}

\begin{Lem}
Assume $\Gamma^{\p}$ is torsion-free. The natural map
$$\HH_c^{d-s}(\Gamma^{\p}, \Div^{\dagger}(\Dr_{B_\p}^{\qd,[s,r]}))\too \HH_c^{d-s}(\Gamma^{\p}, \Div^{\dagger}(\Dr_{B_\p}))$$
is an isomorphism for every $s\geq 1$.
Moreover, there exists an up to sign canonical isomorphism 
$$\HH_c^{d-s}(\Gamma^{\p}, \Div^{\dagger}(\Gamma^{\p}x))= \Lambda^{s}(\Z^{t})$$
for every $x \in X^{\qd,t}$, $1\leq t\leq r$.
\end{Lem}
\begin{proof}
All claims are vacuous if $d=0$.
Thus, we may assume that $d\geq 1$.
We want to reduce the statement to one on arithmetic subgroups as in the proof of Proposition \ref{new}.
Note that in this case the orbits of the action of $\Gamma^\p$ on the Bruhat--Tits tree are easy to describe:
the group $G$ is $F$-simple and simply-connected.
Since $G(F\otimes_\Q \R)$ is not compact, strong approximation implies that $\Gamma^{\p}$ is dense in $G_\p$ (see for example \cite{PRR}, Theorem 7.12).
So for every vertex $v$ of $\mathcal{T}$ we have an isomorphism
$$\Coind_{K_{\p,v}}^{G_\p}\Div(\red^{-1}(v))\cong \Coind_{\Gamma_v}^{\Gamma^{\p}}\Div(\red^{-1}(v)),$$
where $\Gamma_v$ is the intersection of $\Gamma^{\p}$ with $K_{\p,v}$.
An analogous statement holds if one replaces the vertex $v$ by an edge $e$ of $\mathcal{T}$.
By Shapiro's Lemma for cohomology combined with Lemma \ref{coind} the claim can be deduced from the lemma below.
\end{proof}

\begin{Lem}
Let $\Gamma\subseteq G(F)$ be a torsion-free arithmetic congruence subgroup and $X\subseteq \Dr_p$ a $\Gamma$-stable subset.
The natural map
$$\HH_c^{d-s}(\Gamma, \Div(X^{\qd,[s,r]}))\too \HH_c^{d-s}(\Gamma, \Div(X))$$
is an isomorphism for every $s\geq 1$.
Moreover, there exists an up to sign canonical isomorphism
$$\HH_c^{d-s}(\Gamma, \Div(X^{\qd,t}))=\bigoplus_{\Gamma\backslash X^{\qd,t}} \Lambda^{s}(\Z^{t})$$
for every $1\leq t\leq r$.
\end{Lem}
\begin{proof}
By \cite{BS}, Theorem 11.4.2, the group $\Gamma$ is a Bieri-Eckmann duality group (see for example \cite{KB}, Chapter VIII.10, for more details on duality groups).
If $G$ is non-split, the dualizing module of $\Gamma$ is $\Z$ and its cohomological dimension is $d$.
In the case $G$ is split, the dualizing module of $\Gamma$ is $\Delta_0$ and its cohomological dimension is $d-1$.
In any case, we have a canonical isomorphism
$$\HH^{d-i}_c(\Gamma,M)\xrightarrow{\cong}\HH_{i}(\Gamma,M)$$
for every $\Z[\Gamma]$-module $M$ and therefore, in particular,
$$\HH^{d-i}_c(\Gamma,\Div(X))\cong\HH_{i}(\Gamma,\Div(X)).$$
Let us choose a representative $x\in X$ of every $\Gamma$-orbit $[x]=\Gamma.x$ in $X$ and write $\Gamma_x$ for the stabilizer of $x$ in $\Gamma$.
Then, we have an isomorphism
$$\Div(X)\cong\bigoplus_{[x]\in \Gamma\backslash X} \Z[\Gamma_x\backslash\Gamma]$$
of $\Z[\Gamma]$-modules, which by Shapiro's Lemma for homology implies that
\begin{align*}
\HH_{i}(\Gamma,\Div(X)
&\cong\bigoplus_{[x]\in \Gamma\backslash X}\HH_{i}(\Gamma,\Z[\Gamma_x\backslash\Gamma])\\
&\cong\bigoplus_{[x]\in \Gamma\backslash X}\HH_{i}(\Gamma_x,\Z).
\end{align*}
Thus, the higher cohomology vanishes if $\Gamma_x$ is trivial.

Let us now assume that $\Gamma_x$ is non-trivial.
Since $\Gamma_x$ is torsion-free, it follows that $G(F)_x$ has a non-torsion element.
Lemma \ref{groups} then implies that $x$ is an element of $X^{\qd}$ and the stabilizer $G(F)_x$ is isomorphic to the group of elements of relative norm one in the quadratic extension $F(x)/F$.
Thus, the group $\Gamma_x=\Gamma\cap G(F)_x$ is isomorphic to a torsion-free finite index subgroup of the elements of relative norm one in the ring of integers of $F(x)$.
By Dirichlet's unit theorem $\Gamma_x$ is a free abelian group of rank $t=S_{\infty,sp}(x)$.

By the standard computation of the homology of free abelian groups (see for example \cite{KB}, Chapter V, Theorem 6.4) there exists an isomorphism
$$\HH_i(\Gamma_x)\cong \Lambda^{i} (\Z^t)$$
for every $i\geq 0$, which proves the claim.
\end{proof}

%% file: Mer-Cocycles.tex
\subsection{Rigid meromorphic cocycles}\label{Cocycles}
We call an element $J$ of $\HH_c^{t}(\Gamma^{\p},\mathcal{M}^{\times})$ a rigid meromorphic cocycle of degree $i$ (and level $\Gamma^{\p}$).

The following theorem is a generalization of \cite{DV}, Theorem 1 (see also Theorem 2.12 of \textit{loc.~cit.}).
\begin{Thm}\label{sndthm}
Let $\Gamma^{\p}\subseteq G(F)$ be a $\p$-arithmetic congruence subgroup.
\begin{enumerate}[(i)]
\item\label{sndthm0} Let $J$ be a rigid meromorphic cocycle of degree $d-r$.
Then $\supp(\Div(J))$ is a finite union of $\Gamma^{p}$-orbits in $\Dr_{B_\p}^{\qd,r}$.
\item\label{sndthm1} For every $\Gamma^{\p}$-stable subset $X\subseteq \Dr_{B_\p}$ the canonical map
$$\HH_c^{d-r}(\Gamma^{\p},\mathcal{M}^{\times}\hspace{-0.2em}(X))\too\HH_c^{d-r}(\Gamma^{\p},\mathcal{M}^{\times})$$
is injective.
Its image consists of all rigid meromorphic cocycles $J$ of degree $d-r$ with $\supp(\Div(J))\subseteq X$.
In particular, the map
$$\HH_c^{d-r}(\Gamma^{\p},\mathcal{M}^{\times}\hspace{-0.2em}(\Dr_{B_\p}^{\qd,r}))\xlongrightarrow{\cong}\HH_c^{d-r}(\Gamma^{\p},\mathcal{M}^{\times})$$
is an isomorphism.
\item\label{sndthm2}
The canonical map
$$\HH^{i}_c(\Gamma^{\p},\mathcal{A}^{\times})\xrightarrow{\cong}\HH^{i}_c(\Gamma^{\p},\mathcal{M}^{\times})$$
is an isomorphism for all $i<d-r=r_1(B)+2\cdot r_2.$
\item
In the range $d-r\leq i \leq d-1$ the quotient
$$\HH^{i}_c(\Gamma^{\p},\mathcal{M}^{\times})/\HH^{i}_c(\Gamma^{\p},\mathcal{A}^{\times})$$
has countably infinite $\Z$-rank.
\end{enumerate}
\end{Thm}
\begin{proof}
The first claim is an immediate consequence of Proposition \ref{Theorem} \eqref{T1}.

By definition we have a short exact sequence
$$0 \too \mathcal{M}^{\times}\hspace{-0.2em}(X)\too \mathcal{M}^{\times} \too \Div^{\dagger}(\Dr_{B_\p} \hspace{-0.2em}\smallsetminus X) \too 0.$$
By Proposition \ref{Theorem} \eqref{T1} we get an exact sequence in cohomology
$$0 \too \HH^{d-r}(\Gamma^{\p},\mathcal{M}^{\times}\hspace{-0.2em}(X))\too \HH^{d-r}(\Gamma^{\p},\mathcal{M}^{\times}) \too \HH^{d-r}(\Gamma^{\p},\Div^{\dagger}(\Dr_{B_\p}  \hspace{-0.2em}\smallsetminus X))
$$
This proves the second claim.

The third claim follows similarly by considering the long exact sequence associated to the short exact sequence of Proposition \ref{divseq}. 

Via Proposition \ref{first} and Lemma \ref{second} one sees that $\HH^{i}(\Gamma^{\p},\mathcal{A}^{\times})$ is a torsion-module for an appropriate Hecke-algebra (e.g.~one takes the Hecke algebra generated by all Hecke operators away from $\p$, the ramification set of $B$ and the primes $\q$ such that $\Gamma^\p$ is not maximal at $\q$.)
But it can be easily checked that the torsion-free part of the module $\HH^{i}_c(\Gamma^{\p},\Div^{\dagger}(\Dr_{B_\p}))$ is torsion-free over the Hecke algebra.
This implies the last claim.
\end{proof}

\begin{Rem}
Often one is only interested in projective meromorphic cocycles, i.e.~cohomology classes with values in $\mathcal{M}^{\times}\hspace{-0.2em}/\hspace{0.1em}\C_p^{\times}$.
The space of projective meromorphic cocycles is very closely related to the space of divisor values cohomology classes:
using Proposition \ref{first} one can deduce that $\HH^{i}_c(\Gamma^{\p},\mathcal{A}^{\times}\hspace{-0.2em}/\hspace{0.1em}\C_p^{\times})$ is finitely generated for all $i\geq 0$ and, thus, the map
$$\HH^{i}_c(\Gamma^{\p},\mathcal{M}^{\times}\hspace{-0.2em}/\hspace{0.1em}\C_p^{\times})\too \HH^{i}_c(\Gamma^{\p},\Div^{\dagger}(\Dr_{B_\p}))$$ has finitely generated kernel and cokernel for all $i\geq 0$.
\end{Rem}

%% file: Mer-Moduli.tex
\subsection{Rigid meromorphic singular moduli}\label{Moduli}
Let $\Gamma^{\p}\subseteq G(F)$ be a $\p$-arithmetic congruence subgroup and $x\in X^{\qd,t}$ a quadratic point.
We assume for the moment that the stabilizer $\Gamma^{\p}_x$ of $x$ in $\Gamma^{\p}$ is torsion-free and, thus, a free abelian group of rank $t$.
The point $x$ defines a $\Gamma^{\p}_x$-equivariant map
$$f_x\colon \Z \too \Div(\Gamma^{\p}x)\subseteq \Div(\Dr_{B_\p})$$ and, therefore, a homomorphism
$$f_{x,\ast}\colon \HH_{t}(\Gamma^{\p}_x, \Z)\too \HH_{t}(\Gamma^{\p}_x, \Div(\Gamma^{\p}x)).$$
The homology group on the left hand side is isomorphic to $\Z$.
We define $c_x$ to be the image of a generator.
(Thus, $c_x$ is unique up to a sign.)

Now let $J$ be an element of $\HH_c^{t}(\Gamma^{\p},\mathcal{M}^{\times}\hspace{-0.2em}(X))$ for some $\Gamma^{\p}$-stable subset $X\subseteq \Dr_{B_\p}$.
Let $\tilde{J}$ be the image of $J$ in $\HH^{t}(\Gamma^{\p},\mathcal{M}^{\times}\hspace{-0.2em}(X))$ and $\res_x(\tilde{J})$ its restriction to $\Gamma_x^{\p}$.
If $x\notin X$, we can define the value of $J$ at $x$ as
$$J(x)=\res_x(\tilde{J})\cap c_x \in \C_p^{\times}.$$

The maximal $t$ one may choose is $r=r_1(B)+r_2$.
But the divisor of a meromorphic cocycle of degree $t$ is trivial unless $t\geq r_1(B)+2\cdot r_2$ by part \eqref{T1} of Proposition \ref{Theorem}.
So, if one wants to consider rigid meromorphic singular moduli as in \cite{DV} one is forced to assume that the field $F$ is totally real.
In that case, the interesting - by which we mean genuinely meromorphic - cocycles, which we can evaluate at quadratic points live in the middle degree cohomology, i.e.~$t=r=r_1(B)$.

So let us assume that $F$ is totally real.
In that case we can define the value $J(x)$ for any $J\in \HH^{r}(\Gamma^{\p},\mathcal{M}^{\times})$ and any quadratic point $x\in X^{\qd,r}$ satisfying $x\notin \supp(\Div(J))$:
the torsion of the stabilizer group $\Gamma^{\p}_x$ has order at most two.
If the torsion is non-trivial, we fix a torsion-free (and thus free abelian) subgroup $H_x\subseteq \Gamma^{\p}_x$ of index 2.
Otherwise we simply put $H_x=\Gamma^{\p}_x.$
As before, we consider the class $c_x\in \HH_{r}(H_x, \Div(\Gamma^{\p}x)).$
By Theorem \ref{sndthm} \eqref{sndthm1} we can view $J$ as an element of $\HH^{r}(\Gamma^{\p},\mathcal{M}^{\times}\hspace{-0.2em}(\supp(\Div(J)))).$
In particular, the value of $J$ at $x$
$$J(x)=\tilde{J}\cap c_x \in \C_p^{\times}$$
is well-defined.
If $\Gamma^{\p}_x$ is torsion-free, then the set $\left\{J(x),J(x)^{-1}\right\}$ is independent of the choice of generator.
If $\Gamma^{\p}_x$ is not torsion-free, then the set $\left\{J(x)^2,J(x)^{-2}\right\}$ is independent of all choices.

\begin{Rem}
Analogues of the algebraicity conjecture (see \cite{DV}, Conjecture 3.5) for rigid meromorphic singular moduli in the case of quaternion algebras over totally real number fields were formulated and numerically verified by Guitart, Masdeu and Xarles (cf.~\cite{GMX}).
\end{Rem}